\documentclass{amsart}
\usepackage{amssymb}
\usepackage{amsthm}
\newtheorem{theor}{Theorem}[section] 

\theoremstyle{definition} \newtheorem{defin}{Definition}[section]
\newtheorem{ex}{Example}[section]
\theoremstyle{remark} \newtheorem{rem}{Remark}[section]
\newcommand{\pn}{\par\noindent} \newcommand{\pmn}{\par\medskip\noindent}

\begin{document}
\title{Even and odd plane labelled bipartite trees}
\author{Yury Kochetkov}
\date{}
\begin{abstract} Let $T(n,m)$ be the set of plane labelled bipartite trees
with $n$ white vertices and $m$ --- black. If the number $m+n$ of
vertices is even, then the set $T(n,m)$ is a union of two
disjoined subsets --- subset of "even" trees and subset of "odd"
trees. This partition has a clear geometric meaning.
\end{abstract} \email{yukochetkov@hse.ru, yuyukochetkov@gmail.com}
\maketitle

\section{Introduction}
\pn A \emph{plane tree} is a tree embedded into plane. A
\emph{bipartite tree} is a tree with vertices colored in two
colors black and white in such way that adjacent vertices have
different colors. The \emph{passport} of a bipartite tree is the
non increasing sequence of degrees of its white vertices and the
non increasing sequence of degrees of its black vertices.

\begin{ex}\pn
The tree \parbox{3cm}{\begin{picture}(85,40)
\put(15,5){\circle*{3}} \put(15,35){\circle*{3}}
\put(30,20){\circle{4}} \put(50,20){\circle*{3}}
\put(70,20){\circle{4}} \put(15,5){\line(1,1){14}}
\put(15,35){\line(1,-1){14}} \put(32,20){\line(1,0){36}}
\end{picture}} has passport $\langle 3,1\,|\,2,1,1\rangle$.
\end{ex} \pmn The \emph{dual passport} of a bipartite tree is an
expression of the form $\bigl(1^{i_1} 2^{i_2}\ldots\,|\, 1^{j_1}
2^{j_2}\ldots\bigr)$, where $i_1,i_2,\ldots$ are numbers of white
vertices of degree 1, 2 and so on, and $j_1,j_2,\ldots$ are
numbers of black vertices of degree 1, 2 and so on. The dual
passport of the tree in the above example is $\bigl(1^1 3^1\,|\,
1^2 2^1\bigr)$. \pmn A labelled graph is a graph, where each
vertex has a label and these labels are pairwise distinct. We will
consider plane labelled bipartite trees. Here the set of labels of
white vertices and and the set of labels of black ones are
disjoint. We will label white vertices as $v_1,v_2,\ldots$ and
black as $u_1,u_2,\ldots$. \pmn The first problem is about
enumeration: what is the number of plane labelled bipartite trees
with $n$ white vertices and $m$ black ones? \pmn \textbf{Theorem
2.1.} \emph{The number of plane labelled bipartite trees with $n$
white vertices and $m$ black ones is }
$$\dfrac{\bigl((m+n-2)!\bigr)^2}{(n-1)!\cdot (m-1)!}\,.$$

\begin{rem} In what follows we will use notation $(n,m)$-tree to
denote a bipartite trees with $n$ white vertices and $m$ black
ones. \end{rem} \pmn If the number of vertices is even, then the
set of all plane labelled bipartite trees is the union of two
disjoint subsets: the subset of "even" trees and the subset of
"odd" ones. We correspond to a plane bipartite labelled tree $T$
its \emph{invariant} $i(T)=0,1$ (see Definition 3.1). A tree $T$
is even, if $i(T)=0$, and odd in the opposite case. \pmn An
\emph{elementary movement} is a three-step procedure
\begin{enumerate}
    \item we disengage a leaf, i.e. a vertex of degree one with the outgoing
    edge, from the adjacent vertex of the opposite color;
    \item then we move this leaf along the tree till the next
    meeting of a vertex of the opposite color (it can be the same
    vertex, to which the leaf was attached before the move);
    \item we attach the leaf to this vertex.
\end{enumerate}
\pmn \textbf{Theorem 2.1.} \emph{An elementary movement changes
the parity of a tree}.

\begin{rem} All this can be considered as a generalization of
even/odd partition of the set of permutations. \end{rem}

\section{Enumeration}
\pn If $T$ is a plane $(n,m)$-tree, then it generates
$$\dfrac{n!\cdot m!}{\#\text{Aut}(T)} \eqno(1)$$ labelled trees, where
$\#\text{Aut}(T)$ is the order of group of automorphisms of $T$.
Let $M$ be a set of all plane $(n,m)$-trees with a fixed passport
$\Pi=\langle k_1,k_2,\ldots\,|\,l_1,l_2,\ldots\rangle$ and let
$\widetilde{\Pi}=\bigl(1^{i_1} 2^{i_2}\ldots\,|\, 1^{j_1}
2^{j_2}\ldots\bigr)$ be the dual passport. The Goulden-Jackson
theorem \cite{GJ} states that
$$\sum_{T\in M}\frac{1}{\#{\rm Aut}(T)}=\dfrac{(n-1)!\cdot (m-1)!}
{i_1!\cdot\ldots\cdot i_s!\cdot j_1!\cdot\ldots\cdot j_s!}\,,$$
where $s=n+m-1$ is the number of edges.

\begin{theor} The number of plane labelled $(n,m)$-trees is
$$\dfrac{\bigl((n+m-2)!\bigr)^2}{(n-1)!\cdot (m-1)!}\,.$$
\end{theor}
\pmn \emph{Proof.} According to (1) and the Goulden-Jackson
theorem, the number of plane labelled $(n,m)$-trees with the given
passport $\Pi=\langle k_1,k_2,\ldots\,|\,l_1,l_2,\ldots\rangle$ is
$$n!\cdot m!\cdot \dfrac {(n-1)!\cdot (m-1)!}
{i_1!\cdot\ldots\cdot i_s!\cdot j_1!\cdot\ldots\cdot
j_s!}\,.\eqno(2)$$ Numbers $k_1,k_2,\ldots$ constitute a partition
of $s=n+m-1$ of the length exactly $n$ and numbers
$l_1,l_2,\ldots$ constitute a partition of $s$ of the length
exactly $m$. Thus, we must sum (2) by this partitions, or sum
$$n!\cdot \dfrac{(n-1)!}{i_1!\cdot\ldots\cdot i_s!}$$ by
partitions of the length $n$, sum
$$m!\cdot\dfrac{(m-1)!}{j_1!\cdot\ldots\cdot
j_s!}$$ by partitions of the length $m$, and multiply. \pmn Each
partition $(k_1,k_2,\ldots,k_n)=1^{i_1} 2^{i_2}\ldots$ of $s$ of
the length exactly $n$ generates
$$\dfrac{n!}{(i_1)!\cdot (i_2)!\cdot \ldots}$$ solutions of the
equation $x_1+\ldots+x_n=s$, where each solution is a permutation
of numbers $k_1,k_2,\ldots,k_n$. But the number of all positive
integral solutions of this equation is $\binom{s-1}{n-1}$, thus
the double sum by all partitions of lengths $n$ and $m$ is
$$n!\cdot m!\cdot \binom{s-1}{n-1}\cdot
\binom{s-1}{m-1}\cdot \frac{1}{n\cdot
m}=\dfrac{\bigl((n+m-2)!\bigr)^2}{(n-1)!\cdot (m-1)!}\,.\quad
\qed$$

\section{Invariant}
\pn Let $v_1,\ldots,v_n$ be labels of white vertices of a plane
labelled $(n,m)$-tree $T$ and $u_1,\ldots,u_m$ be labels of black
vertices. Some white vertex $v_i$ will be the root vertex and some
edge $e$, outgoing from $v_i$, will be the root edge. We start the
counterclockwise going around of $T$, beginning from $v_i$,
keeping $e$ to the left and in the process of this going we
generate the string $c(T)$ of labels and closing brackets: when we
meet some vertex for the first time we write its label and when we
meet it for the last time we write ")". In $c(T)$ the number of
labels is equal to the number of brackets and in each left segment
the number of labels is not less than the number of brackets.
Thus, in the string $c(T)$ we have a unique correspondence between
labels and brackets.

\begin{ex} Let $v_1$ be the root vertex of the tree $T$
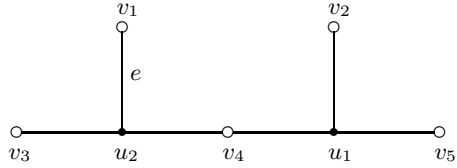
\[\begin{picture}(160,70) \put(0,15){\circle{4}}
\put(40,15){\circle*{3}} \put(40,55){\circle{4}}
\put(80,15){\circle{4}} \put(120,15){\circle*{3}}
\put(120,55){\circle{4}} \put(160,15){\circle{4}}
\put(2,15){\line(1,0){76}} \put(40,15){\line(0,1){38}}
\put(82,15){\line(1,0){76}} \put(120,15){\line(0,1){38}}
\put(-3,5){\small $v_3$} \put(37,5){\small $u_2$}
\put(78,5){\small $v_4$} \put(118,5){\small $u_1$}
\put(158,5){\small $v_5$} \put(38,60){\small $v_1$}
\put(118,60){\small $v_2$} \put(43,35){\small $e$}
\end{picture}\] and $e$ be its root edge. Then
$c(T)=v_1u_2v_3)v_4u_1v_5)v_2)))))$.
\end{ex}

\begin{ex} Let $c(T)=v_1u_2v_2u_1))v_3)v_4u_3))))$. Then
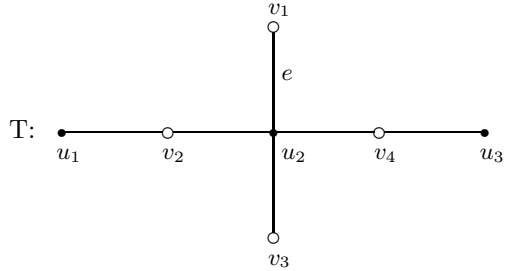
\[\begin{picture}(180,110) \put(20,55){\circle*{3}}
\put(60,55){\circle{4}} \put(100,55){\circle*{3}}
\put(100,15){\circle{4}} \put(100,95){\circle{4}}
\put(140,55){\circle{4}} \put(180,55){\circle*{3}}
\put(20,55){\line(1,0){38}} \put(62,55){\line(1,0){76}}
\put(142,55){\line(1,0){38}} \put(100,17){\line(0,1){76}}
\put(18,45){\small $u_1$} \put(58,45){\small $v_2$}
\put(103,45){\small $u_2$} \put(98,5){\small $v_3$}
\put(98,100){\small $v_1$} \put(138,45){\small $v_4$}
\put(178,45){\small $u_3$} \put(103,75){\small $e$} \put(0,53){T:}
\end{picture}\]
\end{ex}

\begin{rem} From here we will assume that a tree has an even
number of vertices. \end{rem}

\begin{defin} Let $T$ be a plane labelled $(n,m)$-tree with the root vertex
$v_i$ and the root edge $e$ and let $c(T)$ be the corresponding
string. Also let
\begin{itemize}
    \item $a$ be the number of inversions in vertices $v$, i.e.
    the number of cases, when $v_k$ is before $v_l$ in $c(T)$, but $k>l$;
    \item $b$ be the analogously defined number of inversions in
    vertices $u$;
    \item $c$ be the number of cases, when some $u$ is before some $v$
    in $c(T)$;
    \item $d$ be the number of cases, when a closing bracket is before
    some label;
    \item $e=\frac{|n-m|}{2}$.
\end{itemize} Let $i(T)\equiv \bigl(a+b+c+\frac{d+e}{2}\bigr)\text{ mod } 2$.
A tree $T$ will be called \emph{even}, if $i(T)=0$, and \emph{odd}
in the opposite case.
\end{defin}

\begin{theor} The invariant does not depend on a choice of root
edge. \end{theor}

\begin{proof} Let $k$ be the degree of root vertex $v_i$.
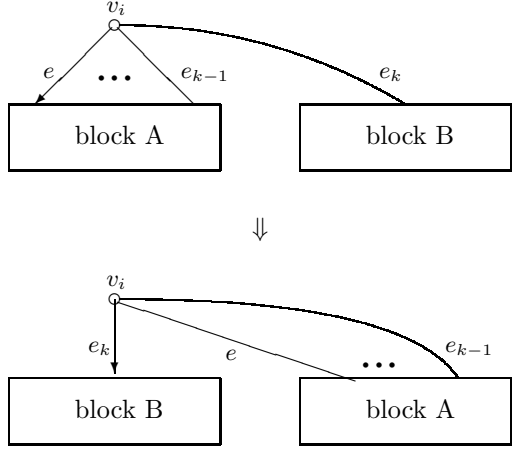
\[\begin{picture}(190,80) \put(0,5){\line(1,0){80}}
\put(0,5){\line(0,1){25}} \put(25,15){block A}
\put(0,30){\line(1,0){80}} \put(80,5){\line(0,1){25}}
\put(40,60){\circle{4}} \put(37,65){\small $v_i$}
\put(39,59){\vector(-1,-1){29}} \put(13,40){\small $e$}
\put(41,59){\line(1,-1){29}} \put(65,40){\small $e_{k-1}$}
\multiput(35,40)(5,0){3}{\circle*{2}} \put(110,5){\line(1,0){80}}
\put(110,5){\line(0,1){25}} \put(110,30){\line(1,0){80}}
\put(190,5){\line(0,1){25}} \put(135,15){block B}
\qbezier(42,60)(100,60)(150,30) \put(140,40){\small $e_k$}
\end{picture}\]
$$\Downarrow$$
\[\begin{picture}(190,75) \put(0,5){\line(1,0){80}}
\put(0,5){\line(0,1){25}} \put(25,15){block B}
\put(0,30){\line(1,0){80}} \put(80,5){\line(0,1){25}}
\put(40,60){\circle{4}} \put(37,65){\small $v_i$}
\put(40,60){\vector(0,-1){28}} \put(30,40){\small $e_k$}
\put(110,5){\line(1,0){80}} \put(110,5){\line(0,1){25}}
\put(110,30){\line(1,0){80}} \put(190,5){\line(0,1){25}}
\put(135,15){block A} \put(41,59){\line(3,-1){90}}
\qbezier(42,60)(150,60)(170,30) \put(82,37){\small $e$}
\put(165,40){\small $e_{k-1}$}
\multiput(135,35)(5,0){3}{\circle*{2}}
\end{picture}\] We will study the change of invariant induced by
the change of a root edge, demonstrated in the figure above. \pmn
Assume that there are
\begin{itemize}
    \item $k_A$ "white" labels in block A and $k_B$ "white" labels
    in block B, $k_A+k_B=n-1$;
    \item $l_A$ "black" labels in block A and $l_B$ "black" labels
    in block B, $l_A+l_B=m$;
    \item $x$ inversions in "white" labels between blocks A and
    B;
    \item $y$ inversions in "black" labels between blocks A and
    B.
\end{itemize} The change of root edge
\begin{itemize}
    \item  decreases the number of inversions in
    white labels by $x$, but increases it by $k_Ak_B-x$;
    \item decreases the number of inversions in
    black labels by $y$, but increases it by $l_Al_B-y$;
    \item decreases the number of inversions in
    white and black labels by $l_Ak_B$, but increases it by $l_Bk_A$;
    \item does not change the number of inversions in labels and
    brackets.
\end{itemize} Thus, we must find the parity of the number
$k_Ak_B+l_Al_B-l_Ak_B+l_Bk_A$. Let $z=k_A+l_A$, then
\begin{multline*}k_Ak_B+l_Al_B-l_Ak_B+l_Bk_A\equiv
k_Ak_B+l_Al_B+l_Ak_B+l_Bk_A=\\=
(k_A+l_A)(k_B+l_B)=z(n+m-1-z)=\\=(m+n)z-z(z+1)\equiv 0\text{ mod }
2\end{multline*}
\end{proof}

\begin{theor} The invariant does not depend on a choice of root
vertex. \end{theor}

\begin{proof} Let the root vertex be changed from $v_i$ to $v_j\,$:
\[\begin{picture}(280,105) \put(0,5){\line(1,0){40}}
\put(0,5){\line(0,1){15}} \put(0,20){\line(1,0){40}}
\put(40,5){\line(0,1){15}} \put(5,10){\small block C}
\put(20,20){\line(0,1){13}} \put(20,35){\circle{4}}
\put(8,35){\small $v_j$} \put(21,36){\line(1,1){58}}
\put(50,65){\circle*{3}} \put(80,95){\circle{4}}
\put(37,65){\small $u_p$} \put(69,95){\small $v_i$}
\put(50,65){\line(1,-1){15}} \put(45,50){\line(1,0){40}}
\put(45,50){\line(0,-1){15}} \put(45,35){\line(1,0){40}}
\put(85,35){\line(0,1){15}} \put(50,40){\small block B}
\put(81,94){\line(1,-1){14}} \put(75,80){\line(1,0){40}}
\put(75,80){\line(0,-1){15}} \put(75,65){\line(1,0){40}}
\put(115,65){\line(0,1){15}} \put(80,70){\small block A}

\put(140,40){$\Rightarrow$}

\put(165,80){\line(1,0){40}} \put(165,80){\line(0,-1){15}}
\put(165,65){\line(1,0){40}} \put(205,65){\line(0,1){15}}
\put(170,70){\small block C} \put(185,80){\line(1,1){14}}
\put(200,95){\circle{4}} \put(188,95){\small $v_j$}
\put(201,94){\line(1,-1){58}} \put(230,65){\circle*{3}}
\put(260,35){\circle{4}} \put(235,65){\small $u_p$}
\put(265,35){\small $v_i$} \put(260,33){\line(0,-1){13}}
\put(240,20){\line(1,0){40}} \put(240,20){\line(0,-1){15}}
\put(240,5){\line(1,0){40}} \put(280,5){\line(0,1){15}}
\put(245,10){\small block A} \put(230,65){\line(-1,-1){15}}
\put(195,50){\line(1,0){40}} \put(195,50){\line(0,-1){15}}
\put(195,35){\line(1,0){40}} \put(235,35){\line(0,1){15}}
\put(200,40){\small block B} \end{picture}\] Then the string is
changed in the following way:
$$v_iu_pv_j\,C)B)A) \Rightarrow v_j\,C\,u_p\,B\,v_i\,A)))$$ As
above we will assume that blocks A, B, C contain $k_A$ "white"
labels and $l_A$ "black" labels, $k_B$ "white" labels and $l_B$
"black" labels, $k_C$ "white" labels and $l_C$ "black" labels,
respectively. Then (if we do not take into account even terms) the
invariant is changed by
\begin{multline*}[\pm 1+1+k_C+l_C+(k_C+l_C)/2+k_B+l_B+(k_B+l_B)/2]+
\\ + [-1-k_C+l_C+(k_C+l_C)/2] +\\+
[-(k_A+l_A)/2-(k_B+l_B)/2-(k_A+l_A)/2].\end{multline*} Here terms
in the first square brackets are generated by movement of $v_i$,
in the second square brackets --- by movement of $u_p$ and in the
third square brackets --- by movement of two closing brackets in
the string $c(T)$. Thus, the change is
$$\pm 1-k_A-l_A+k_B+l_B+k_C+3l_C\equiv n+m-2 \equiv 0\text{ mod }
2.$$ \end{proof}

\section{Movements}
\pn A leaf is a vertex of degree one with the edge outgoing from
it.
\begin{defin} Let $T$ be a plane labelled bipartite tree. A
\emph{movement} is a 3-step procedure: a) we disengage a black
(white) leaf $A$ from white (black) vertex $B$ to which this leaf
is attached; b) we move the leaf $A$ around $T$ clockwise or
counter clockwise to a white (black) vertex $C$ (it is possible,
that $C$ is $B$); c) we attach the leaf $A$ to $C$. A movement of
a black (white) leaf is \emph{even}, if it bypassed an even number
of black (white) vertices, and odd in the opposite case. A
movement of a black (white) leaf will be called \emph{elementary},
if it bypassed one black (white) vertex.
\end{defin}

\begin{ex} In the figure below we see the movement of black leaf
"$\alpha$" from white vertex "a" to white vertex "b".
\[\begin{picture}(330,70) \put(20,15){\circle{4}}
\put(20,55){\circle*{3}} \put(70,15){\circle*{3}}
\put(70,55){\circle{4}} \put(120,15){\circle{4}}
\put(0,15){\line(1,0){18}} \put(22,15){\line(1,0){96}}
\put(122,15){\line(1,0){18}} \put(20,17){\line(0,1){38}}
\put(70,15){\line(0,1){38}} \put(18,5){\small a}
\put(118,5){\small b} \put(69,59){\scriptsize c} \put(68,5){\small
$\beta$} \put(18,58){\scriptsize $\alpha$}
\qbezier[30](25,20)(45,20)(65,20) \put(37,20){$\longrightarrow$}
\qbezier[30](65,20)(65,45)(65,65) \put(61,30){$\uparrow$}
\put(61,50){$\uparrow$} \qbezier[8](65,65)(70,65)(75,65)
\put(66,65){$\to$} \qbezier[30](75,65)(75,43)(75,20)
\put(76,50){$\downarrow$} \put(76,30){$\downarrow$}
\qbezier[30](75,20)(95,20)(115,20) \put(87,20){$\longrightarrow$}

\put(150,13){$\Rightarrow$}

\put(210,15){\circle{4}} \put(310,55){\circle*{3}}
\put(260,15){\circle*{3}} \put(260,55){\circle{4}}
\put(310,15){\circle{4}} \put(190,15){\line(1,0){18}}
\put(212,15){\line(1,0){96}} \put(312,15){\line(1,0){18}}
\put(310,17){\line(0,1){38}} \put(260,15){\line(0,1){38}}
\put(208,5){\small a} \put(308,5){\small b}
\put(259,59){\scriptsize c} \put(258,5){\small $\beta$}
\put(308,58){\scriptsize $\alpha$}
\end{picture}\] "$\alpha$" bypasses black vertex "$\beta$" twice, so
this movement is even. \end{ex}

\begin{rem} A reason for the number of vertices to be even is that
otherwise a movement of a leaf clockwise and counterclockwise to
the same final position is even in one case and odd --- in
another. \end{rem}

\begin{theor} An elementary movement changes the parity of a tree.
\end{theor}

\begin{proof} We will check all types of elementary movements.
\begin{itemize}
    \item A counterclockwise movement of a white leaf $v_j$ increases the distance
    between the root vertex $v_i$ and $v_j$ by 2: at first the leaf is attached to
    the black vertex $u_p$, then it moves to the black vertex
    $u_q$, bypassing the white vertex $v_s$.
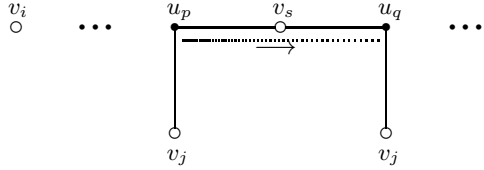
    \[\begin{picture}(180,70) \put(0,55){\circle{4}}
    \put(-3,60){\small $v_i$} \put(60,55){\circle*{3}}
    \put(60,15){\circle{4}} \put(100,55){\circle{4}}
    \put(140,55){\circle*{3}} \put(140,15){\circle{4}}
    \put(57,5){\small $v_j$}
    \multiput(25,55)(5,0){3}{\circle*{2}}
    \put(60,55){\line(1,0){38}} \put(60,55){\line(0,-1){38}}
    \put(102,55){\line(1,0){38}} \put(140,55){\line(0,-1){38}}
    \put(137,5){\small $v_j$} \put(57,60){\small $u_p$}
    \put(97,60){\small $v_s$} \put(137,60){\small $u_q$}
    \multiput(165,55)(5,0){3}{\circle*{2}}
    \qbezier[50](63,50)(80,50)(137,50)
    \put(90,45){$\longrightarrow$} \end{picture}\] This movement
    changes the string in the following way:
    $$\cdots\,v_j)v_su_q\,\cdots \quad\Rightarrow\quad\cdots\,
    v_su_qv_j)\,\cdots$$ It changes the number of inversions
    in white labels by 1, the number of inversions in white and
    black labels by 1 and the number of inversions in labels and
    brackets by 2, i.e. the invariant changes by $\pm 1+1+1$ ---
    by an odd number.
    \item A counterclockwise movement of a white leaf $v_j$ does
    not change the distance between the root vertex $v_i$ and $v_j$  --- the
    movement at first decreases this distance and then increases it.
    In the beginning the leaf is attached to the black vertex $u_p$, then it moves
    to the black vertex $u_q$, bypassing the white vertex $v_s$.
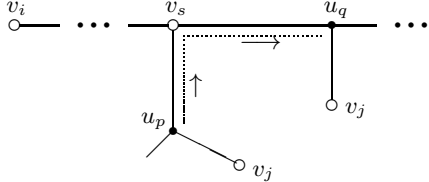
    \[\begin{picture}(160,80) \put(0,65){\circle{4}}
    \put(-3,70){\small $v_i$} \put(2,65){\line(1,0){15}}
    \put(60,65){\circle{4}} \put(57,70){\small $v_s$}
    \put(58,65){\line(-1,0){15}}
    \multiput(25,65)(5,0){3}{\circle*{2}}
    \put(62,65){\line(1,0){75}} \put(120,65){\circle*{3}}
    \put(117,70){\small $u_q$}
    \multiput(145,65)(5,0){3}{\circle*{2}}
    \put(60,25){\circle*{3}} \put(60,25){\line(0,1){38}}
    \put(50,15){\line(1,1){10}} \put(60,25){\line(2,-1){24}}
    \put(48,28){\small $u_p$} \put(85,12){\circle{4}}
    \put(90,8){\small $v_j$} \put(120,35){\circle{4}}
    \put(120,37){\line(0,1){28}} \put(125,32){\small $v_j$}
    \qbezier[25](64,28)(64,40)(64,61)
    \qbezier[30](64,61)(90,61)(116,61) \put(66,40){$\uparrow$}
    \put(85,56){$\longrightarrow$} \end{picture}\] This movement
    changes the string in the following way:
    $$\cdots\,v_j))u_q\,\cdots \quad\Rightarrow\quad\cdots\,
    )u_qv_j)\,\cdots$$ Actually, it only changes the number of
    inversions in white labels and black labels by 1.
    \item A counterclockwise movement of a white leaf $v_j$ does
    not change the distance between the root vertex $v_i$ and $v_j$  --- the
    movement at first increases this distance and then decreases
    it. This movement in essence is an interchange of positions of
    two neighboring white leaves and thus only change the number
    of inversions in white labels by one.
    \item A counterclockwise movement of a white leaf $v_j$ decreases
    the distance between the root vertex $v_i$ and $v_j$ by 2: at first
    the leaf is attached to
    the black vertex $u_p$, then it moves to the black vertex
    $u_q$, bypassing the white vertex $v_s$.
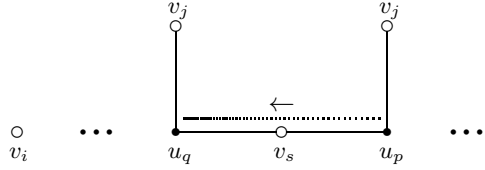
    \[\begin{picture}(180,70) \put(0,15){\circle{4}}
    \put(-3,5){\small $v_i$} \put(60,15){\circle*{3}}
    \put(60,55){\circle{4}} \put(100,15){\circle{4}}
    \put(140,15){\circle*{3}} \put(140,55){\circle{4}}
    \put(57,60){\small $v_j$}
    \multiput(25,15)(5,0){3}{\circle*{2}}
    \put(60,15){\line(1,0){38}} \put(60,15){\line(0,1){38}}
    \put(102,15){\line(1,0){38}} \put(140,15){\line(0,1){38}}
    \put(137,60){\small $v_j$} \put(57,5){\small $u_q$}
    \put(97,5){\small $v_s$} \put(137,5){\small $u_p$}
    \multiput(165,15)(5,0){3}{\circle*{2}}
    \qbezier[50](63,20)(80,20)(137,20)
    \put(95,23){$\leftarrow$} \end{picture}\] This movement
    changes the string in the following way:
    $$\cdots\,v_j)))\,\cdots \quad\Rightarrow\quad\cdots\,
    ))v_j)\,\cdots$$ It only changes the number of
    inversions in brackets and labels by 2, i.e. it changes the
    invariant by 1.
    \item Now let the root vertex $v_i$ be of degree one and it
    makes a movement
    
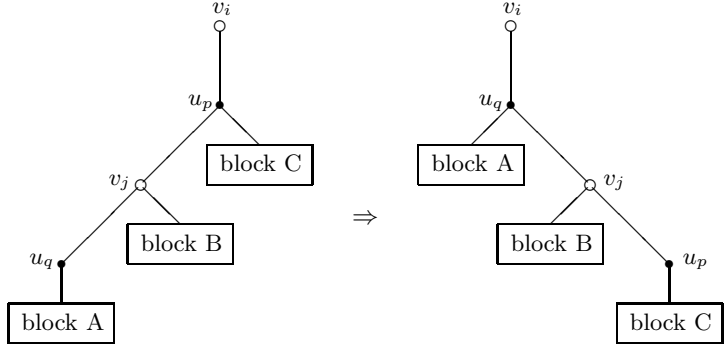
\[\begin{picture}(300,140) \put(0,5){\line(1,0){40}}
    \put(0,5){\line(0,1){15}} \put(0,20){\line(1,0){40}}
    \put(40,5){\line(0,1){15}} \put(5,10){\small block A}
    \put(20,35){\circle*{3}} \put(20,35){\line(0,-1){15}}
    \put(8,35){\small $u_q$} \put(50,65){\circle{4}}
    \put(20,35){\line(1,1){29}} \put(38,65){\small $v_j$}
    \put(45,50){\line(1,0){40}} \put(45,50){\line(0,-1){15}}
    \put(45,35){\line(1,0){40}} \put(85,35){\line(0,1){15}}
    \put(50,40){\small block B} \put(51,64){\line(1,-1){14}}
    \put(80,95){\circle*{3}} \put(80,95){\line(-1,-1){29}}
    \put(68,95){\small $u_p$} \put(75,80){\line(1,0){40}}
    \put(75,80){\line(0,-1){15}} \put(75,65){\line(1,0){40}}
    \put(115,65){\line(0,1){15}} \put(80,70){\small block C}
    \put(80,95){\line(1,-1){15}} \put(80,125){\circle{4}}
    \put(80,95){\line(0,1){28}} \put(77,130){\small $v_i$}

    \put(130,50){$\Rightarrow$}

    \put(155,80){\line(1,0){40}}
    \put(155,80){\line(0,-1){15}} \put(155,65){\line(1,0){40}}
    \put(195,65){\line(0,1){15}} \put(160,70){\small block A}
    \put(175,80){\line(1,1){15}} \put(190,95){\circle*{3}}
    \put(190,95){\line(0,1){28}} \put(190,125){\circle{4}}
    \put(187,130){\small $v_i$} \put(190,95){\line(1,-1){29}}
    \put(220,65){\circle{4}} \put(178,95){\small $u_q$}
    \put(221,64){\line(1,-1){29}} \put(250,35){\circle*{3}}
    \put(225,65){\small $v_j$} \put(185,50){\line(1,0){40}}
    \put(185,50){\line(0,-1){15}} \put(185,35){\line(1,0){40}}
    \put(225,35){\line(0,1){15}} \put(190,40){\small block B}
    \put(205,50){\line(1,1){14}} \put(255,35){\small $u_p$}
    \put(250,35){\line(0,-1){15}} \put(230,20){\line(1,0){40}}
    \put(230,20){\line(0,-1){15}} \put(230,5){\line(1,0){40}}
    \put(270,5){\line(0,1){15}} \put(235,10){\small block C}
    \end{picture}\] The string is changed in the following
    way:
    $$v_iu_pv_ju_qA)B)C)) \Rightarrow v_iu_qAv_jBu_pC))))$$ Let us assume
    that
    \begin{itemize}
        \item block A contains $k_A$ white labels ($x$ of them
        precede $v_j$), $l_A$ black labels ($y$ of them precede
        $u_p$) and $k_A+l_A$ brackets;
        \item block B contains $k_B$ white labels, $l_B$ black labels
        ($z$ of them precede $u_p$) and $k_B+l_B$ brackets;
        \item block C contains $k_B$ white labels and $l_B$ black
        labels.
    \end{itemize} The movement of $u_p$ in the string changes the
    invariant in the following way:
    \begin{itemize}
        \item $u_pv_j\to v_j\ldots u_p\,$: $-1$;
        \item $u_p\ldots u_q\to u_q\ldots u_p\,$: $\pm 1$;
        \item $u_p\ldots A\to A\ldots u_p\,$: $-k_A+l_A-2y+(k_A+l_A)/2$;
        \item $u_p\ldots B\to B u_p\,$: $-k_B+l_B-2z+(k_B+l_B)/2$.
    \end{itemize} $-k_A/2+3l_A/2-k_B/2+3l_B/2-2y-2z-1 \pm 1$ in total. \pmn The movement of $v_j$
    in the string changes the invariant in the following way:
    \begin{itemize}
        \item $v_ju_q\to u_q\ldots v_j\,$: $1$;
        \item $v_j\ldots A\to A v_j\,$: $k_A-2x+l_A+(k_A+l_A)/2$.
    \end{itemize} $3k_A/2+3l_A/2+1-2x$ in total. \pmn
    Movement of brackets makes two bypasses of $C$ and one bypass
    of $B$, thus this movement changes the invariant by
    $-(k_B+l_B)/2-k_C-l_C$. Thus, the total change is
    \begin{align*}
    k_A+3l_A-k_B+l_B-&k_C-l_C-2x-2y-2z\pm 1\equiv \\ &\equiv
    k_A+l_A+k_B+l_B+k_C+l_C+1 \text{ mod 2}.\end{align*} It remains to note
    that $k_A+l_A+k_B+l_B+k_C+l_C+1=n+m-3$ --- odd number.
    \item All other cases are obvious. The analysis of the black
    leaf movement is the same as the analysis of white leaf
    movement.
\end{itemize}
\end{proof}

\begin{rem} In \cite{K1} and \cite{K2} it was proved that the set of
plane bipartite weighted trees with six vertices and the given
lists of white and black weights in generic case is a union of two
subsets. An analytically defined invariant $I(T)=\pm\sqrt d$,
where $d$ is the product of weights of vertices, determines the
belonging of a tree $T$ to this or that subset. In generic case a
weighted tree is a labelled tree with constraints. The study of
geometrical properties of invariant $I(T)$ is the origin of this
work.\end{rem}

\vspace{5mm}

\begin{thebibliography}{99}

\bibitem{GJ} Goulden I.P. and Jackson D.M., The combinatorial relationship between
trees, cacti and certain connection coefficients for the symmetric
group, European J. Combin., 1992, V.13, p. 357-365.
\bibitem{K1} Kochetkov Yu., Anti-Vandermonde systems and plane
trees, Fuct. Anal. Appl., 36:3 (2002), 240-243.
\bibitem{K2} Kochetkov Yu., Enumeration of one class of plane
weighted trees, J. Math. Sci., 209:2 (2015), 282-291.
\end{thebibliography}
\end{document}